\documentclass[12pt]{article}

\usepackage{amsmath,amssymb, amsthm,url,tikz}

\usetikzlibrary{intersections,arrows}

\newcommand{\im}{\mathrm{im }\,}

\newcommand{\cA}{\mathcal{A}}
\newcommand{\cB}{\mathcal{B}}
\newcommand{\cC}{\mathcal{C}}
\newcommand{\cG}{\mathcal{G}}
\newcommand{\cR}{\mathcal{R}}
\newcommand{\cS}{\mathcal{S}}

\newcommand{\bZ}{\mathbb{Z}}

\newtheorem{theorem}{Theorem}
\newtheorem{lemma}{Lemma}
\newtheorem{question}{Question}

\theoremstyle{definition}

\begin{document}
\title{Triangulations of the sphere, bitrades and abelian groups}
\author{Simon R. Blackburn\thanks{Department of Mathematics, Royal
    Holloway, University of London, Egham, Surrey TW20 0EX.}\,\, and
  Thomas A. McCourt\thanks{Heilbronn Institute for Mathematical
    Research, School of Mathematics, University of
    Bristol, University Walk, Bristol BS1 1TW.}}
\maketitle

\begin{abstract}
  Let $\cG$ be a triangulation of the sphere with vertex set $V$, such
  that the faces of the triangulation are properly coloured black and
  white. Motivated by applications in the theory of bitrades, Cavenagh
  and Wanless defined $\cA_W$ to be the abelian group generated by the
  set $V$, with relations $r+c+s=0$ for all white triangles with
  vertices $r$, $c$ and $s$. The group $\cA_B$ can be defined
  similarly, using black triangles.

  The paper shows that $\cA_W$ and $\cA_B$ are isomorphic, thus
  establishing the truth of a well-known conjecture of Cavenagh and
  Wanless. Connections are made between the structure of $\cA_W$ and
  the theory of asymmetric Laplacians of finite directed graphs, and
  weaker results for orientable surfaces of higher genus are
  given. The relevance of the group $\cA_W$ to the understanding of
  the embeddings of a partial latin square in an abelian group is also
  explained.
\end{abstract}

\section{Introduction}
\label{sec:introduction}

Let $\cG$ be a triangulation of the sphere, where each triangle
is coloured either black or white, and where no two triangles of the
same colour share a common edge, i.e. $\cG$ is a (properly) face
2-coloured triangulation. Let $W$ be the set of white triangles, and
let $B$ be the set of black triangles. We assume that $\cG$ is finite
(indeed, all objects in this paper are finite unless indicated otherwise).

Heawood~\cite{Heawood} showed that there is a (proper) vertex
$3$-colouring of $\cG$: see for example the proof of Theorem~4.11 in
Wilson~\cite[Page 38]{Wilson}. Let $R$, $C$ and $S$ be the colour
classes, so $V=R\cup C\cup S$ in this colouring and every triangle
contains exactly one vertex from each of $R$, $C$ and $S$.

Define an abelian group $\cA_W$ using the set of white triangles as
follows: $\cA_W$ is the abelian group with
generating set $V$, subject to the relations
$\{r+c+s=0:\{r,c,s\}\in W\}$. Define $\cA_B$ from the set of black
triangles in $\cG$ in a similar way: $\cA_B$ has the same
generating set $V$ as $\cA_W$, but relations $\{r+c+s=0:\{r,c,s\}\in
B\}$.

We prove the following theorem:

\begin{theorem}
\label{thm:isomorphism}
For any triangulation $\cG$ as above, $\cA_W\cong \cA_B$.
\end{theorem}

This establishes the truth of a well-known conjecture by Cavenagh and
Wanless~\cite[Conjecture~1]{CavenaghWanless}; see the 17th edition of the Kourovka
Notebook~\cite[Problem~17.35]{Kourovka17} or the list of problems
from the 21st British Combinatorial Conference~\cite[Problem~524,
BCC22.21]{Cameron}. The motivation for this problem came from a
conjecture on latin bitrades due to Cavenagh and
Dr\'apal~\cite{CavenaghDrapal}; we provide more details in
Section~\ref{sec:motivation}. (Incidentally, the origins of this
problem in the theory of latin trades explains our choice of notation for
the colour classes of our triangulation: the triangle $\{r,c,s\}$
corresponds to an entry of symbol $s$ in row $r$ and column $c$ of a
partial latin square.)

Let $G$ be the graph embedded in the sphere to form $\cG$ and let $V$
be the vertex set of $G$.  For an edge $e$ of $G$, we write $\psi(e)$
for the set of its two end vertices. Every triangle is uniquely
determined by its colour together with the set of three distinct
incident vertices. Note that it is possible for there to be two
triangles of different colours with the same associated triple of
vertices; the simplest example is when we have two faces that are
identified with opposite hemispheres, with a triangle on the equator
between them, and Figure~\ref{fig:nasty_triangulation} gives a more
complicated example. This example also shows that an edge is not
necessarily determined by its end vertices (so $G$ is not necessarily
simple); moreover, the rotation at a vertex (in other words, the boundary of
the union of all triangles containing a vertex) is not necessarily a
simple cycle.  Triangulations in which the rotation at every vertex is
a simple cycle are sometimes known as \emph{piecewise-linear
  triangulations}.  Since all triangulations arising from latin
bitrades are piecewise linear, our Theorem~\ref{thm:isomorphism} is,
in fact, a little more general than the application to latin bitrades needs.

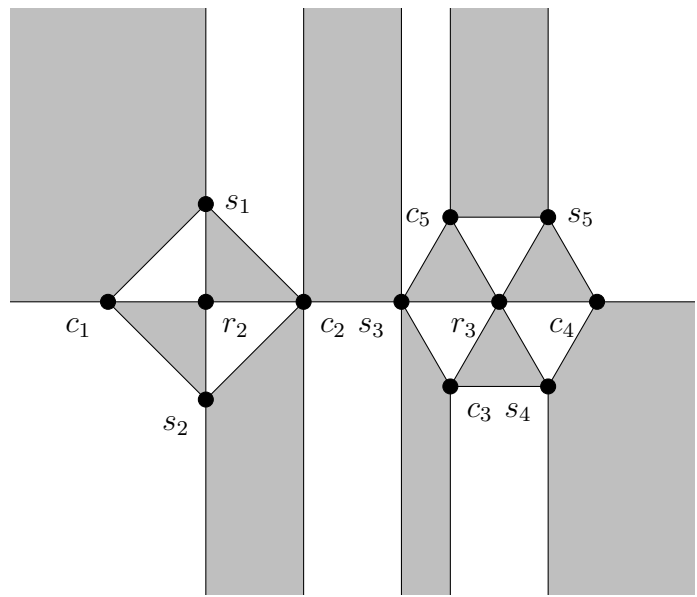
\begin{figure}
\begin{center}
\begin{tikzpicture}[fill=gray!50, scale=1.3,
vertex/.style={circle,inner sep=2,fill=black,draw}]

\clip (-5,-3) rectangle (2,3);

\coordinate (r3) at (0,0);
\coordinate (c4) at (0:1);
\coordinate (s5) at (60:1);
\coordinate (c5) at (120:1);
\coordinate (s3) at (180:1);
\coordinate (c3) at (240:1);
\coordinate (s4) at (300:1);

\coordinate (c2) at (-2,0);
\coordinate (r2) at (-3,0);
\coordinate (c1) at (-4,0);
\coordinate (s1) at (-3,1);
\coordinate (s2) at (-3,-1);

\filldraw (r3) -- (s3) -- (c5) -- cycle;
\filldraw (r3) -- (s5) -- (c4) -- cycle;
\filldraw (r3) -- (c3) -- (s4) -- cycle;
\filldraw (r2) -- (s1) -- (c2) -- cycle;
\filldraw (r2) -- (c1) -- (s2) -- cycle;

\filldraw (c2) -- (s3) -- +(0,5) -- (-2,5) -- cycle; 
\filldraw (c5) -- (s5) -- ++(0,3) -- ++(-1,0) -- cycle;

\filldraw (s3) -- (c3) -- ++(0,-4) -- ++(120:1) -- cycle;
\filldraw (c2) -- (s2) -- ++(0,-4) -- ++(1,1) -- cycle;

\filldraw (c1) -- (s1)-- ++(0,6) -- ++(-6,0) -- ++(0,-7) -- cycle;
\filldraw (c4) -- (s4) -- ++(0,-4) -- ++(60:1) -- ++(4,0) -- ++(0,4) -- cycle;

\node at (r3) [vertex,label=south west:$r_3\,$]{};
\node at (c5) [vertex,label=west:$c_5$]{};
\node at (s5) [vertex,label=east:$s_5$]{};
\node at (c4) [vertex,label=south west:$c_4\,$]{};
\node at (s4) [vertex,label=south west:$s_4$]{};
\node at (c3) [vertex,label=south east:$c_3$]{};
\node at (s3) [vertex,label=south west:$s_3$]{};

\node at (r2) [vertex,label=south east:$r_2$]{};
\node at (c2) [vertex,label=south east:$c_2$]{};
\node at (s2) [vertex,label=south west:$s_2$]{};
\node at (c1) [vertex,label=south west:$c_1$]{};
\node at (s1) [vertex,label=east:$s_1$]{};
\end{tikzpicture}
\end{center}
\caption{A triangulation. One node ($r_1$) has been placed at infinity.}
\label{fig:nasty_triangulation}
\end{figure}

The structure of the rest of this paper is as follows. In
Section~\ref{sec:motivation}, we discuss how our problem arises in the
study of spherical latin bitrades. The results and terminology in this
section are not needed to understand the main results of this paper,
but motivate these results and are needed to understand some of our comments on
possible future directions at the end of this paper. In
Section~\ref{sec:preliminary}, we prove a preliminary lemma on
connected planar bipartite graphs. Theorem~\ref{thm:isomorphism} is
proved in Section~\ref{sec:proof_of_theorem}. We establish more
information about the structure of the group $\cA_W$ in
Section~\ref{sec:torsion}, connecting this group with the theory of
asymmetric Laplacians on directed graphs. In particular, we use this
theory to provide examples of groups $\cA_W$ of maximal possible rank,
and of exponential order. Finally, in Section~\ref{sec:comments}, we
give examples to show that Theorem~\ref{thm:isomorphism} cannot hold
for general surfaces, and comment on the form of $\cA_W$ and $\cA_B$
in this more general case. We also provide a selection of open
problems in the area.

\section{Background and motivation}
\label{sec:motivation}

Let $G$ be a graph (not necessarily simple) with $v$ vertices. If the
edges of $G$ can be partitioned into isomorphic copies of $K_3$, then
such a partition is called a \emph{partial triple system of order
  $v$}; moreover if the maximum number of edges in $G$ that have the
same end points is $\lambda$, then the partial triple system is said
to have index $\lambda$, we write PTS($v,\lambda$).

Now consider a partial triple system, PTS($v,\lambda$) say, that is
vertex 3-colourable and let $n$ be the size of the largest colour
class. The system, together with a vertex 3-colouring, is called a
\emph{partial transveral design of block size three and index
  $\lambda$}; we write PTD$_\lambda(3,n)$. Moreover, if each pair of
vertices from distinct colour classes occur in precisely $\lambda$
triples (hence, $v\equiv 0\,(\text{mod } 3)$ and the graph $G$ partitioned
into copies of $K_3$ is the $\lambda$-fold complete tripartite graph
$\lambda K_{v/3,v/3,v/3}$), then it is said to be a \emph{transveral
  design of block size three and index $\lambda$}.  Finally, if
$\lambda=1$ we say that the (partial) transversal design is called a
\emph{(partial) latin square}. The \emph{order} of a partial latin
square is the size of the largest colour class (in other words, $n$). See Part
III of \cite{CRC} for a survey of results on latin squares.

Let $P$ be a partial latin square, of order $n$, with vertex colouring
represented by a partition of the vertices into the sets $R$ (the
rows), $C$ (the columns) and $S$ (the symbols). Note that $R$, $C$ and
$S$ are by definition pairwise disjoint. We can relabel the
vertices in $R$ (respectively, $C$ and $S$) as $r_i$ (respectively,
$c_j$ and $s_k$), indexed by a set of size $|R|$ (respectively, $|C|$
and $|S|$). Then any triple in $P$ is of the form $\{r_i, c_j,s_k\}$,
and symbol $k$ can be thought of as occurring in row $i$ and column
$j$ of an $n\times n$ array. Thus a partial latin square has at most
one occurence of each symbol in any row or column, and
each cell of the array has at most one symbol assigned to it. (If we
replace `at most one' in the preceeding sentence with `exactly one' we
have a latin square.)

A \emph{latin bitrade} is a pair $(W,B)$ of nonempty partial latin
squares such that for each triple $\{r_i,c_j,s_k\}\in W$ (respectively
$B$) there exist unique $r_{i'}\neq r_i$, $c_{j'}\neq c_j$ and
$s_{k'}\neq s_k$ such that
\[
\{\{r_{i'},c_j,s_k\},\{r_i,c_{j'},s_k\},\{r_i,c_j,s_{k'}\}\}\subset
B\text{ (respectively }W\text{).}
\]
Each of $W$ and $B$ is called a \emph{latin trade}, and we say that
$W$ is the \emph{mate} of $B$ (and vice versa). The \emph{size} of the
latin bitrade is $|W|$ (or $|B|$). It is possible for some trade, $W$
say, to occur in many distinct bitrades.

Although latin trades have been studied implicitly in many different
papers (generally when considering the difference between or the
intersection of two latin squares of the same order) the first
explicit appearance of latin trades in the literature is in a paper by
Dr\'apal and Kepka \cite{DrapalKepka}, where they are referred to as
\emph{exchangeable partial groupoids}. Subsequently to this they have
been extensively studied (see \cite{CavenaghSurvey} for a survey of
results up until 2008).

In this paper we are motivated by the topological properties of latin
trades.  Let $\cG$ be a 2-cell embedding of a graph $G$ in a surface
$\cS$. If every face of $G$ is a 3-cycle, then we say that $\cG$ is a
\emph{triangulation} of $\cS$.  Now, consider a face 2-coloured
triangulation of some surface, then each colour class corresponds to a
partial triple system. The two partial triple systems $W$ and $B$ are
said to be \emph{biembedded} in the surface. If the triple systems
are disjoint partial latin squares, then $(W,B)$ forms a latin
bitrade, and we say that there is a biembedding of the bitrade $(W,B)$
into $\cS$. In particular when the biembedding is in the sphere we say
that $(W,B)$ is a \emph{spherical latin bitrade}. We say that a latin
bitrade is \emph{decomposable} if there exist proper subsets
$W'\subseteq W$ and $B'\subseteq B$ such that $(W',B')$ is a
bitrade. If $(W,B)$ is an indecomposable latin bitrade of size $t$,
then $|R|+|C|+|S|\leq t+2$, with equality exactly when $(W,B)$ is
spherical. Thus the property of being spherical can be motivated
purely combinatorially. In recent years biembeddings of bitrades have
been extensively studied, see \cite{Drapal}, \cite{Cavenagh}, \cite
{Drapal2}, \cite{Drapal3} and \cite{CavenaghWanless}.

Two partial latin squares are said to be \emph{isotopic} if they
are equal up to relabelling of their sets of rows, columns and
symbols.  A partial latin square $P$ is said to \emph{embed in an
  abelian group $A$} if there exist injective maps $f_1:R\rightarrow A$,
$f_2:C\rightarrow A$ and $f_3:S\rightarrow A$ such that
$f_1(r)+f_2(c)=f_3(s)$ for all $\{r,c,s\}\in P$. In other words, $P$
embeds in an abelian group~$A$ if and only if it is isotopic to a
partial latin subsquare $P'$ of the multiplication table of $A$. (Of
course, this definition can be easily generalised to nonabelian
groups, indeed to general quasigroups.)

\begin{lemma}
\label{lem:ab_embedding}
Let $V=R\cup C\cup S$, where the sets $R$, $C$ and $S$ are disjoint. 
A partial latin square $P$ embeds in an abelian group $A$ if and only
if there exists a function $f:V\rightarrow A$ that is injective when
restricted to each of $R$, $C$ and $S$ and is such that
$f(r)+f(c)+f(s)=0$ for all $\{r,c,s\}\in P$.
\end{lemma}
\begin{proof}
  Using the notation in the definition of an embedding above, define
  $f$ by $f|_{R}=f_1$, $f|_C=f_2$ and $f|_S=-f_3$.
\end{proof}

From now on, we refer to a function $f$ satisfying the conditions in
Lemma~\ref{lem:ab_embedding} as an \emph{embedding} of $P$ in $A$. If
the image of $f$ does not generate $A$, clearly $P$ may be embedded in
a subgroup of $A$.  We say that an abelian group $A$ is a
\emph{minimal abelian representation} for the partial latin square $P$
if $P$ embeds in $A$, and the image of $f$ generates $A$ for all
embeddings $f$ of $P$ in $A$.

Note that the notion of an embedding of a partial transveral design
$P$ in an abelian group only makes sense when the design is in fact a
partial latin square, since the definition of an embedding implies
that a triple in $P$ is determined uniquely by any two of its
components.

In \cite{CavenaghDrapal} Cavenagh and Dr\'apal asked: ``Can the
individual trades in any biembedding of a latin bitrade be embedded
into the operation table of an abelian group? If this is not true in
general is it true for spherical latin bitrades?''  This question for
spherical latin bitrades was resolved positively by Canvenagh and
Wanless in \cite{CavenaghWanless} (see also Dr\'apal, H\"am\"al\"ainen
and Kala~\cite{DHK}), who also gave examples of
biembeddings of latin bitrades $(W,B)$ into surfaces of higher genus
such that $W$ does not embed into any group. Their work motivated
their conjecture that $\cA_W\cong\cA_B$ for spherical latin bitrades
$(W,B)$, where $\cA_W$ and $\cA_B$ are as defined in the
introduction. Note that the definition of the group $\cA_W$ makes
sense for any partial transversal design $W$. However, the group
$\cA_W$ is particularly relevant once we know that $W$ is a partial
latin square that embeds in some abelian group, because of the
`universal' property of $\cA_W$ for embeddings of $W$, given by the
following theorem.
(A more general statement, for both abelian and nonabelian groups, 
was proved in \cite{newDK}.)

\begin{theorem}
\label{thm:universal_ab_theorem}
Let $W$ be a partial latin square, and suppose there exists an
embedding $f:V\rightarrow A$, where $A$ is an abelian
group. Then the natural map $g:V\rightarrow \cA_W$ is
an embedding of $W$ in $\cA_W$, and there is a homomorphism
$h:\cA_W\rightarrow  A$ such that $f=hg$.
\end{theorem}
\begin{proof}
  Let $F(V)$ be the free abelian group on the set $V$. There is a
  unique homomorphism $\zeta:F(V)\rightarrow A$ such that
  $\zeta|_V=f$. The elements $r+c+s$ for $\{r,c,s\}\in W$ lie in the
  kernel of $\zeta$, since $f$ is an embedding of $W$ into $A$. Hence
  $\zeta$ induces a homomorphism $h:\cA_W\rightarrow A$, and it is
  easy to see that $f=hg$. This implies that $g$ is injective
  when restricted to $R$, $C$ or $S$, and so $g$ is an embedding, as
  required.
\end{proof}
In particular, Theorem~\ref{thm:universal_ab_theorem} shows that any
minimal abelian representation of $W$ is a quotient of $\cA_W$. In
fact, more is true. It is not hard to see that there is a homomorphism
$\nu:\cA_W\rightarrow \bZ\oplus \bZ$ such that $\nu(r)=(1,0)$,
$\nu(c)=(0,1)$ and $\nu(s)=(-1,-1)$ for $r\in R$, $c\in C$ and $s\in
S$. Defining $\cC_W$ to be the kernel of $\nu$, we see that
$\cA_W\cong \bZ\oplus \bZ\oplus \cC_W$. (Here we are using the facts
that~$\nu$ is surjective and $\bZ\oplus \bZ$ is free abelian; see
Rotman~\cite[Corollary~10.21]{Rotman}, for example.) 
Both Dr\'apal et al~\cite{DHK} and 
Cavenagh and
Wanless~\cite{CavenaghWanless} show that $\cC_W$ is finite when
$(W,B)$ is a spherical bitrade, so $\cC_W$ is the torsion subgroup of $\cA_W$. (They consider a group that is
isomorphic to $\cC_W$, which they call $\cA_W^*$.)  We will reprove
this result in Section~\ref{sec:torsion}.

It is worth noting as an aside that Cavenagh and Wanless have
retracted their proof of the statement~\cite[Theorem~7]{CavenaghWanless}
that the torsion subgroups of $\cA_W$ and $\cA_B$ have equal
order. (Their proof of their Theorem~7 assumes, in their notation, that
pairing each white face $f$ with the black face not adjacant to the
vertex $v_f$ produces a bijection between sets of black and white
faces. However, there are simple examples where this is not the case.)
Of course, their Theorem~7 still holds, since it is a corollary of our
Theorem~\ref{thm:isomorphism}.

The following theorem is not stated explicitly by either Cavenagh and Wanless or Dr\'apal et al,
but can easily be proved from their work.
(Although the result can also be obtained as a direct corollary of
earlier work by Dr\'apal and Kepka, Lemma 3.3 in \cite{newDK}, we
believe that the proof we give is more accessible.) 
\begin{theorem}
\label{thm:small_universe}
Let $W$ be a partial latin square that can be embedded in an abelian
group. Define the group $\cC_W$ as above. Then any minimal abelian
representation $A$ of $W$ is a quotient of $\cC_W$. 
\end{theorem}
\begin{proof}
  Let $f:V\rightarrow A$ be an embedding of $W$ in $A$, and let
  $g:V\rightarrow \cA_W$ and $h:\cA_W\rightarrow A$ be defined as in
  Theorem~\ref{thm:universal_ab_theorem}. Now $f(V)$ generates $A$,
  since $A$ is a minimal abelian representation of $W$. Since $h$ is a
  homomorphism and $h(g(V))=f(V)$, we see that $h$ is onto and so (by
  the first isomorphism theorem for groups) $A\cong Q$, where
  $Q=\cA_W/K$ for the subgroup $K:=\ker h$ of $\cA_W$. Without loss of
  generality, we may identify $A$ with $Q$.

Let $\overline{\nu}:Q\rightarrow (\bZ\oplus \bZ)/\nu(K)$ be the
homomorphism induced from the map $\nu$ defined above. Let $r'\in R$,
$c'\in C$ and $s'\in S$ be fixed. Then
$\overline{\nu}(f(R))=(1,0)+\nu(K)=\overline{\nu}(f(r'))$, and similarly
$\overline{\nu}(f(C))=\overline{\nu}(f(c'))$ and
$\overline{\nu}(S)=\overline{\nu}(f(s'))$. Define a new embedding $f'$ of $W$
into $Q$ by
\[
f'(v) =\left\{\begin{array}{cl}
f(v)-f(r')\text{ if }v\in R,\\
f(v)-f(c')\text{ if }v\in C,\\
f(v)-f(s')\text{ if }v\in S.
\end{array}
\right.
\]
Then $f'$ is an embedding of $W$ in $Q$, and
$\overline{\nu}(f'(V))=0$. But then $f'(V)\subseteq \ker
\overline{\nu}$ and so, since $Q$ is a minimal abelian representation of $W$, we
must have $\ker \overline{\nu}=Q$. Thus $\overline{\nu}$ is the all
zero map, and so $\nu(K)=\im\nu$. In particular, there exist
$k_1,k_2\in K$ such that $\nu(k_1)=(1,0)$ and $\nu(k_2)=(0,1)$.
We have that $\cA_W=\langle k_1\rangle \oplus \langle k_2\rangle
\oplus \cC_W$. Setting $N=K\cap \cC_W$, we find that $K=\langle
k_1,k_2,N\rangle$ and so 
\[
Q=\frac{\langle k_1\rangle \oplus \langle k_2\rangle
\oplus \cC_W}{\langle
k_1,k_2,N\rangle}\cong \frac{\cC_W}{N}.
\qedhere
\]
\end{proof}

\section{Connected bipartite planar graphs}
\label{sec:preliminary}

This section proves a key lemma that we need for the proof of
Theorem~\ref{thm:isomorphism} in Section~\ref{sec:proof_of_theorem}.

Let $C$ and $S$ be non-empty finite sets. Let $\Gamma=(C\cup S,E)$ be
a connected bipartite planar graph on the classes $C$ and $S$. Fix an
embedding of $\Gamma$ in the sphere. Suppose there are $k$ faces
$\cR_i$ in this embedding. (When we apply the material in this
section, $C$ and $S$ will be two of the colour classes of the vertex
$3$-colouring of our triangulation referred to in the introduction. If
we write the third colour class as $R:=\{r_1,r_2,\ldots,r_k\}$, then
$\cR_i$ will be the region corresponding to the union of those
triangles, whether black or white, containing $r_i$. So the boundary walk of $\cR_i$ is the rotation at $r_i$.) For each face
$\cR_i$, let $e_{i1},e_{i2},\ldots ,e_{i\ell_i}$ be the edges in the
boundary of $\cR_i$, taken clockwise. Let $u\in C\cup S$. For a set
$X$, we write $F(X)$ for the free abelian group on the set $X$. For an
edge $z\in E$ with end-vertex set $\psi(z)=\{s,c\}$, we define $\tilde{z}=s+c\in
F(C\cup S)$ and $\tilde{E}=\{\tilde{z}:z\in E\}$. This section proves
the following lemma.

\begin{lemma}
\label{lem:cycle_lemma}
Let $\Gamma$ be a connected bipartite planar graph as above. Then $F(C\cup S)=\langle u\rangle \oplus
\langle \tilde{E}\rangle$. We have that $\langle u\rangle\cong
\bZ$. Moreover, the natural map from $F(E)$ to $\langle
\tilde{E}\rangle$ induces the following isomorphism:
\[
\langle \tilde{E}\rangle\cong  F(E)/
  \left\langle\sum_{a=1}^{\ell_i}(-1)^ae_{ia}:i\in\{1,2,\ldots,k\}\right\rangle.
\]
\end{lemma}
\begin{proof}
We claim that $F(C\cup S)$ is generated by $\{u\}\cup \tilde{E}$. 
Let $u'\in C\cup S\setminus\{u\}$. Since $\Gamma$ is connected, there is a path
$u'=u_0,u_1,\ldots,u_{\ell}=u$ from $u'$ to $u$ in $\Gamma$, where 
$z_i\in E$ and $\psi(z_i):=\{u_i,u_{i+1}\}$ for all $i$. But then
\[
u'+(-1)^{\ell-1} u=\sum_{i=0}^{\ell-1}(-1)^i\tilde{z}_i\in\langle
\tilde{E}\rangle.
\]
So $u'\in\langle \{u\}\cup \tilde{E}\rangle$ for all $u'\in C\cup S$, and our claim
follows.

Define a homomorphism $\phi:F(C\cup S)\rightarrow \bZ$ by $\phi(c)=1$
for all $c\in C$ and $\phi(s)=-1$ for all $s\in S$. Since $\Gamma$ is
bipartite with respect to the classes $C$ and $S$, we have that
$\phi(\tilde{E})=0$. But $\phi(u)=\pm 1\not=0$. Thus $\langle
u\rangle\cong\bZ$ and $F(C\cup S)=\langle u\rangle \oplus
\langle \tilde{E}\rangle$.

It remains to establish the last statement of the lemma. Let
$\pi:F(E)\rightarrow \langle\tilde{E}\rangle$ be the natural homomorphism, where
$\pi(x)=\tilde{x}$ for all $x\in E$. It suffices to show that
\begin{equation}
\label{eq:kernel}
\ker\pi=\left\langle\sum_{a=1}^{\ell_i}(-1)^ae_{ia}:i\in\{1,2,\ldots,k\}\right\rangle.
\end{equation}

Clearly $\sum_{a=1}^{\ell_i}(-1)^a\tilde{e}_{ia}=0$ (since the first
and last vertex of the boundary agree, and since $\ell_i$ is even as
$\Gamma$ is bipartite). Thus the right hand side of~\eqref{eq:kernel}
is contained in the left hand side. 

Let $g\in \ker \pi$. We have that $g=\sum_{x\in E}f(x)x$ for some
function $f:E\rightarrow\bZ$. The function $f$ is an integral edge
weighting of $\Gamma$. If we orient the edges of $\Gamma$ from $C$ to
$S$, the condition that $g\in\ker\pi$ is equivalent to the condition
that $f$ lies in the integral flow space of $\Gamma$ (the set of
edge-weightings with zero net flow at each vertex). We claim that the
integral flow space is generated by the set of simple cycles of
$\Gamma$ (flows in which all weights are $\pm 1$, and the set of edges
with non-zero weights form a cycle). This is a standard result (see
\cite[Proposition~7.13]{BondyMurty}, for example), but for the
reader's convenience we establish it here. Take a spanning
tree $S$ for $\Gamma$, and for each edge $x\in \Gamma\setminus S$, let
$C_x$ be the unique simple cycle in $S\cup \{x\}$, oriented so that
$x$ appears with weighting $1$. Given a flow $f$, the flow
$f-\sum_{x\in\Gamma\setminus S}f(x)C_x$ is zero outside $S$, and so
(since $S$ contains no cycles) is identically zero. Thus the set of
simple cycles of the form $C_x$ generate the flow space, as claimed.

Since $\Gamma$ is planar, any simple cycle $\kappa$ is the sum
of the boundary flows of those faces lying inside $\kappa$. Thus the flows
corresponding to boundaries of faces generate $\ker \pi$. But the
boundary of face $\cR_i$ is the element
$\sum_{a=1}^{\ell_i}(-1)^ae_{ia}$, and so the lemma follows.
\end{proof}

\section{Proof of Theorem~\ref{thm:isomorphism}}
\label{sec:proof_of_theorem}

Let $\cG$ be a face 2-coloured finite triangulation of the
sphere.  Let $W$ be the set
of white triangles and
$B$ be the set of black triangles in this triangulation.
As above, let $V$ be the set of vertices of $\cG$, and
write $V=R\cup C\cup S$ where $R$, $C$ and $S$ are the colour classes
of a $3$-colouring of $V$.

Let $R=\{r_1,r_2,\ldots,r_k\}$. Define a $k\times k$ integer matrix
$T=(t_{ij})$ as follows. For $i\not=j$, define $t_{ij}$ to be the
number of edges $e$, where $\psi(e)=\{c,s\}$,
such that $\{r_j,c,s\}\in W$ and
$\{r_i,c,s\}\in B$. Define $t_{ii}=-\sum_{j\not= i}t_{ij}$. Writing
$d_i$ for the number of black triangles containing $r_i$, note that
$d_i+t_{ii}$ is exactly the number of edges $e$, where $\psi(e)=\{c,s\}$,  
that occur in both a triangle $\{r_i,c,s\}\in W$ and a triangle $\{r_i,c,s\}\in B$.

Define a second matrix $T'=(t'_{ij})$ in a similar way, but with the roles of black
and white triangles reversed; so when $i\not=j$ we define $t'_{ij}$ to
be the number of edges $e$, where $\psi(e)=\{c,s\}$,
such that $\{r_j,c,s\}\in B$ and
$\{r_i,c,s\}\in W$, and we define $t'_{ii}:=-\sum_{j\not=i}
t'_{ij}$. Writing $d'_i$ for
the number of white triangles containing $r_i$, we find that
$d'_i+t'_{ii}$ is the number of edges~$e$, where $\psi(e)=\{c,s\}$,
that occur in both a triangle $\{r_i,c,s\}\in B$ and a triangle $\{r_i,c,s\}\in W$.

Clearly $t'_{ij}=t_{ji}$ for $i\not=j$, and also
$d_i+t_{ii}=d'_i+t'_{ii}$. Since the colour of the triangles
containing $r_i$ alternates as we rotate about $r_i$, we have that
$d_i=d'_i$ and hence $t_{ii}=t'_{ii}$. Thus $t_{ij}=t'_{ji}$ for all
$i,j\in\{1,2,\ldots,k\}$ and so $T'$ is the transpose of $T$. One
consequence of this fact is that both the row and the column sums of
$T$ are all zero.

Define $\cB_W$ to be the abelian group generated by elements
$x_1,x_2,\ldots,x_k$, with relations matrix $T$. So
\[
\cB_W=\left\langle x_1,\ldots,x_k \mid \sum_{j=1}^kt_{ij}x_j=0\text{
    for $i\in\{1,2,\ldots,k\}$}\right\rangle.
\]
Similarly, define $\cB_B$ to have relations matrix $T'$:
\[
\cB_B=\left\langle x_1,\ldots,x_k\mid \sum_{j=1}^kt'_{ij}x_j=0\text{
    for $i\in\{1,2,\ldots,k\}$}\right\rangle.
\]
Since the Smith Normal Form of a $k\times k$ matrix is the same as the
Smith Normal Form of its transpose, $\cB_W\cong
\cB_B$. So establishing the following lemma will suffice to prove
Theorem~\ref{thm:isomorphism}:

\begin{lemma}
\label{lem:key_lemma}
We have that $\cA_W\cong \bZ\oplus \cB_W$ and $\cA_B\cong \bZ\oplus \cB_B$.
\end{lemma}

\begin{proof}
We prove that $\cA_W\cong \bZ\oplus \cB_W$. The proof that $\cA_B\cong
\bZ\oplus \cB_B$ is exactly the same, with the roles of black and
white triangles interchanged, so we will omit it.

For each $i\in\{1,2,\ldots,k\}$, choose $c_i\in C$ and $s_i\in S$ so
that $\{r_i,c_i,s_i\}\in W$. We may use the relations
$r_i+c_i+s_i=0$ to eliminate the generators $r_i$ from the
presentation of $\cA_W$. Each relation of the form $r_i+c+s=0$, where
$\{r_i,c,s\}$ is a white triangle, becomes $(c+s)-(c_i+s_i)=0$, and so
we have 
\[
\cA_W=F(C\cup S)/\left\langle U\right\rangle
\]
where $U$ is the set of relations of the form $(c+s)-(c_i+s_i)=0$ for
each white triangle
$\{r_i,c,s\}$.

Define $E$ to be the set of edges in the triangulation $\cG$ whose
endpoints both lie in $C\cup S$. The edges $E$ form a bipartite graph
$\Gamma$ on $C\cup S$ with vertex classes $C$ and $S$. Since the edges
of $E$ form part of a triangulation of the sphere, $\Gamma$ is
planar. Moreover, $\Gamma$ is connected. To see this, note that if two
triangles $\Delta_1,\Delta_2$ in $\cG$ share an edge, then the edges
of $\Delta_1$ and $\Delta_2$ that lie in $E$ have at least one vertex
in common. Since a sphere is connected, any two triangles are
connected by a path consisting of edge-adjacent triangles. So given
two edges $e_1$ and $e_2$, where $\psi(e_1)=\{c,s\}$ and
$\psi(e_2)=\{c',s'\}$, an edge-adjacent path from a triangle
containing $\{c,s\}$ to a triangle containing $\{c',s'\}$ induces a
path from $e_1$ to $e_2$ entirely consisting of edges in~$E$.

Recall from Section~\ref{sec:preliminary} the notation $\tilde{x}=c+s$
for an element in the free abelian group $F(C\cup S)$ corresponding to
an edge $x\in E$, where $\psi(x):=\{c,s\}$.  For
$i\in\{1,2,\ldots,k\}$, let $x_i$ be an edge such that
$\psi(x_i)=\{s_i,c_i\}$ (where $s_i$ and $c_i$ were chosen above)
adjacent to a white triangle $\{r_i,s_i,c_i\}\in W$.  Note that the
relations in $U$ are exactly those of the form
$\tilde{x}-\tilde{x_i}=0$ where there is a white triangle containing
$r_i$ and the edge $x$. Thus, in $\cA_W$, every $\tilde{x}\in
\tilde{E}$ is equal to $\tilde{x}_i$ for some $i$. Let $u\in C\cup S$.
Since all the relations in $U$ lie in $\langle\tilde{E}\rangle$,
Lemma~\ref{lem:cycle_lemma} implies (in the notation there) that
\begin{align}
\nonumber
\cA_W&\cong \langle u\rangle\oplus (\langle\tilde{E}\rangle/\langle U\rangle)\\
\label{eq:key_decomposition}
&\cong\bZ\oplus\left(F(E)/\left\langle
U'\cup\left\{\sum_{a=1}^{\ell_i}(-1)^ae_{ia}:i\in\{1,2,\ldots,k\}\right\}\right\rangle\right),
\end{align}
where $U'$ is a set of relations of the form $x-x_i=0$ for all $x\in
E$ contained in a white triangle involving $r_i$. Recall that the
boundary of region $\cR_i$ consists of the edges
$e_{i1},e_{i2},\ldots,e_{i\ell_i}$, in other words, the rotation at
$r_i$ in $\cG$. Thus the edges in the boundary are all contained in
either a black or a white triangle that is adjacent to $r_i$, and the
colour alternates as we travel round the boundary.  We may choose to
begin the boundary walk at an edge $e_{i1}$ that is contained in a
white triangle that is adjacent to $r_i$.

Each relation of the form $x-x_i=0$ can be used to eliminate one
generator $x$ from our presentation of $\cA_W$. Using all relations of this form,
we can eliminate all generators other than $x_1,x_2,\ldots ,x_k$ from
our presentation.  We claim that each relation of the form 
$\sum_{j=1}^{\ell_i}(-1)^je_{ij}=0$ is transformed to the relation
$\sum_{j=1}^kt_{ij}x_j=0$ during this elimination process. To see
this, note that the edges $e_{ia}$ where $a$ is odd are exactly the edges of
each of the $d'_i$ white triangles containing $r_i$ and so are
replaced by $x_i$. Each such edge appears with a negative sign
in the relation, so $\sum_{a\text{ odd}}(-1)^ae_{ia}=-d'_{i}x_i$. The
remaining edges lie in black triangles containing $r_i$, and
so are replaced by various generators $x_{j_a}$ where $j_a\in\{1,2,\ldots
,k\}$. When $i\not=j$, the
definition of the integers $t_{ij}$ implies that exactly $t_{ij}$ of
the edges are replaced by $x_j$. Moreover, we remarked when we defined
the matrix $T$ that the number of edges $e$, where $\psi(e)=\{c,s\}$,
that are contained in both a triangle
$\{r_i,c,s\}\in B$ and a triangle $\{r_i,c,s\}\in W$ is exactly $d_i+t_{ii}$, where $d_i$ is the number
of black triangles containing $r_i$. These $d_i+t_{ii}$ edges are replaced by
$x_i$, and so
\begin{align*}
\sum_{a=1}^{\ell_i}(-1)^ae_{ia}&=-d'_{i}x_i+\sum_{a\text{
    even}}e_{ia}\\
&=-d'_ix_i+(d_i+t_{ii})x_i+\sum_{j\not=i}t_{ij}x_j\\
&=\sum_{j=1}^kt_{ij}x_j,
\end{align*}
since $d_i=d'_i$. Thus our claim follows.

This change of presentation, together
with~\eqref{eq:key_decomposition} implies that
\begin{align*}
\cA_W&\cong\bZ\oplus
\left(F(\{x_1,x_2,\ldots,x_k\})/\left\langle\sum_{j=1}^kt_{ij}x_j:i\in\{1,2,\ldots,k\}\right\rangle\right)\\
&\cong \bZ\oplus\cB_W,
\end{align*}
and so the lemma is proved.
\end{proof}

\section{The order of the torsion subgroup}
\label{sec:torsion}

The matrix $T$ has determinant zero, since the entries in each row
(and each column) sum to zero. Thus $\cB_W$ is an infinite
group. Cavenagh and Wanless show (in our notation) that $\cB_W\cong
\bZ\oplus \cC_W$, where $\cC_W$ is finite. There is an alternative
proof of this fact via Tutte's Matrix Tree Theorem (see
Tutte~\cite[Chapter~VI]{Tutte}), which we now give; this alternative
proof has the advantage that it relates the order of $\cC_W$ to the
tree number of a certain directed graph $D$, which we now define.

Let $D$ be a directed graph on vertices $z_1,z_2,\ldots ,z_k$, where
we add exactly $t_{ij}$ directed edges from $z_j$ to $z_i$. So $D$ is
a directed multigraph with no loops. In terms of our triangulation
$\cG$, the vertices of $D$ correspond to the regions $\cR_i$ defined
in Section~\ref{sec:preliminary}, with a directed edge from $z_j$ to
$z_i$ for every white triangle in $\cR_j$ that shares an edge with a
black triangle in $\cR_i$. In Tutte's terminology, $-T$ is the
\emph{Kirchoff matrix of the directed graph $D$ with unit
  conductances}. (This matrix is also known as the asymmetric
Laplacian of $D$: see Biggs~\cite{Biggs}.) Since the rows and columns
of $T$ all sum to zero, we see that the in-degree and out-degree of
any vertex $z_i$ of $D$ is zero; in other words, $D$
is~\emph{Eulerian}.

We claim that $D$ is \emph{strongly connected} (in other words, for
all ordered pairs of distinct vertices $z_j$ and $z_i$ there is an
ordered path from $z_j$ to $z_i$). Let $j$ be fixed, and let
$I\subseteq \{1,2,\ldots,k\}$ be defined by $i\in I$ if and only if
there is a directed path from $z_j$ to $z_i$ in $D$. Clearly $j\in I$,
so $I$ is non-empty. Let $\cS:=\bigcup_{i\in I}\cR_i$. Suppose, for a
contradiction, that $I\not=\{1,2,\ldots,k\}$, so $\cS$ is not the
whole sphere. The border of $\cS$ contains a non-trivial cycle
consisting of edges whose end points lie in $C\cup S$. Since the edges
with endpoints in $C\cup S$ form a bipartite graph, this cycle has at
least two edges. Each edge of the cycle is adjacent to a unique
triangle inside $\cS$, and the colour of this triangle alternates as
we move around the cycle. Thus there exists an edge $e$, where
$\psi(e)=\{s,c\}$, on the border of $\cS$ that is adjacent to a white
triangle in $\cS$ and a black triangle in $\cR_{i'}$ for some
$i'\not\in I$. But then there is a directed edge in $D$ from some
$z\in\{z_i:i\in I\}$ to $z_{i'}$, and so $i'\in I$. This contradiction
shows that $I=\{1,2,\ldots,k\}$ and so $D$ is strongly connected as
claimed.

Let $z$ be a
vertex of $D$. An~\emph{arborescence} diverging from $z$ is a directed
subtree of $D$ containing $z$ with all edges oriented away from
$z$. The number $d$ of spanning arborescences in $D$ diverging from
$z$ does not depend on $z$ (see~\cite[Theorem~VI.23]{Tutte}); this
number $d$ is known as the \emph{tree number}~$d$ of~$D$. Since $D$ is
strongly connected and Eulerian, $d>0$. Tutte
shows~\cite[Theorem~VI.28]{Tutte}, as a corollary of the Matrix Tree
Theorem, that $d=\det(-\overline{T})$, where $\overline{T}$ is the
matrix obtained by removing the last row and column from~$T$.

\begin{theorem}
\label{thm:group_order}
Let $d$ be the (non-zero) tree number of the directed graph $D$
defined above. Then $\cB_W\cong \bZ\oplus \cC_W$, where $\cC_W$ is a finite
abelian group of order $d$.
\end{theorem}
\begin{proof}
Since the row sums of the matrix $T$ defining $\cB_W$ are all zero,
the map $\phi:\cB_W\rightarrow \bZ$ given by
$\phi(\sum_{i=1}^ka_ix_i)=\sum_{i=1}^ka_i$ is well defined. Since
$\phi(x_k)=1$, we see that $\phi$ is onto and $x_k$ has infinite
order. Thus
\[
\cB_W=\langle x_k\rangle\oplus \cC_W
\]
where $\cC_W=\ker \phi$ and $\langle x_k\rangle\cong \bZ$.

Let $Q=\cB_W/\langle x_k\rangle\cong \cC_W$, and define
$\overline{x}_i=x_i+\langle x_i\rangle\in Q$. Then $Q$ is generated by
$\overline{x}_1,\overline{x}_2,\ldots,\overline{x}_{k-1}$, subject to
the relations matrix $\overline{T}$ obtained from the matrix $T$ by removing
its last row and column. We argued above that $\det \overline{T}=(-1)^{k-1}
d\not=0$, and so $|\cC_W|=|Q|=d$, as required.
\end{proof}

The structural insight provided by Theorem~\ref{thm:group_order}
enable us to construct examples of triangulations where the group
$\cA_W$ is large:

\begin{theorem}
\label{thm:exponential}
Let $\cG$ be the tiling of the sphere into $k$ regions $\cR_1,
\cR_2,\ldots,\cR_{k}$ given in Figure~\ref{fig:exponential}. \emph{(}In this
tiling, $\cR_1$ shares edges with $\cR_2$ and $\cR_{k}$. We
identify vertices and edges appropriately; so for
example $c_{k+1}=c_{1}$.\emph{)} Then
\[
\cA_W\cong \bZ\oplus\bZ\oplus \cC_W,
\]
where $|\cC_W|=k 2^{k-1}$. Moreover, $\cA_W$ has rank $k+1$.
\end{theorem}

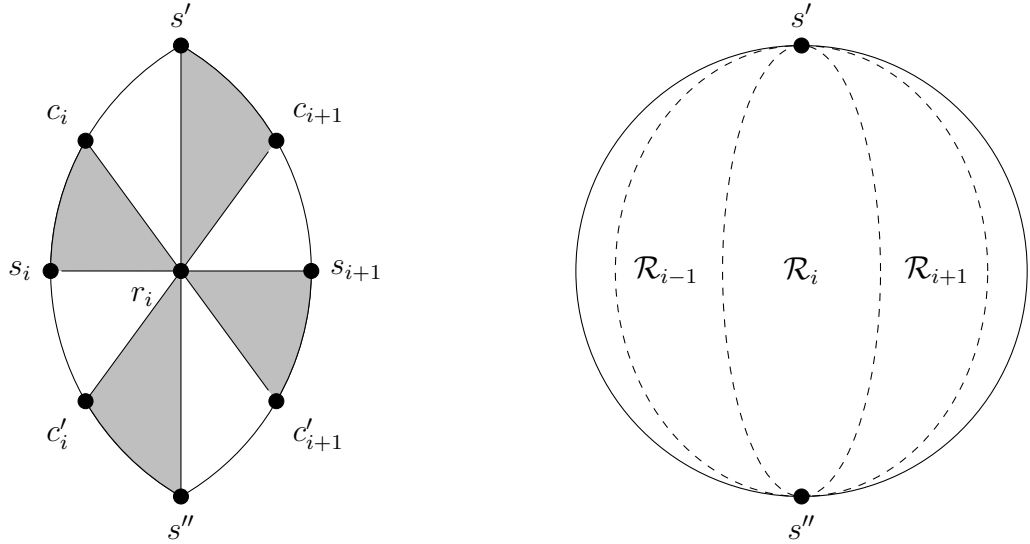
\begin{figure} 
\begin{center}
\begin{tikzpicture}[fill=gray!50, scale=1.5,
vertex/.style={circle,inner sep=2,fill=black,draw}]

\coordinate (r1) at (0,0);

\coordinate (s0) at (0,2);
\coordinate (s00) at (0,-2);

\draw[name path=lpath] (s00) arc (240:120:2.31);
\draw[name path=rpath] (s00) arc (-60:60:2.31);

\path (s00) arc (240:210:2.31) node [name=ci2]{}
arc (210:180:2.31) node [name=si1]{}
arc (180:150:2.31) node [name=ci1]{};

\path (s00) arc (-60:-30:2.31) node [name=cii2]{}
arc (-30:0:2.31) node [name=sii1]{}
arc (0:30:2.31) node [name=cii1]{};

\filldraw (r1) -- (cii1) arc (30:60:2.31) -- (r1);
\filldraw (r1) -- (cii2) arc (-30:0:2.31) -- (r1);
\filldraw (r1) -- (s00) arc (240:210:2.31) -- (r1);
\filldraw (r1) -- (si1) arc (180:150:2.31) -- (r1);

\node at (s0) [vertex,label=north:$s'$]{};
\node at (s00) [vertex,label=south:$s''$]{};
\node at (ci2) [vertex,label=south west:$c'_{i}$]{};
\node at (si1) [vertex,label=west:$s_{i}$]{};
\node at (ci1) [vertex,label=north west:$c_{i}$]{};
\node at (cii1) [vertex,label=north east:$c_{i+1}$]{};
\node at (sii1) [vertex,label=east:$s_{i+1}$]{};
\node at (cii2) [vertex,label=south east:$c'_{i+1}$]{};
\node at (r1) [vertex,label=south west:$r_i\,\,$]{};

\draw (5.5,0) ellipse [x radius=2,y radius=2];
\draw[dashed] (5.5,0) ellipse [x radius=1.65,y radius = 2];
\draw[dashed] (5.5,0) ellipse [x radius=0.7, y radius = 2];

\node at (5.5,2) [vertex,label=north:$s'$]{};
\node at (5.5,-2) [vertex,label=south:$s''$]{};

\node at (4.3,0) {$\mathcal{R}_{i-1}$};
\node at (5.5,0) {$\mathcal{R}_i$};
\node at (6.7,0) {$\mathcal{R}_{i+1}$};
\end{tikzpicture}
\end{center}
\caption{A typical region $\cR_i$, and tiling the sphere with these
  regions.}
\label{fig:exponential}
\end{figure}

\begin{proof}
  Each region $\cR_i$ contains exactly four black triangles; two share
  an edge with a white triangle in $\cR_{i+1}$ and two with white
  triangles in $\cR_{i-1}$ (subscripts being taken cyclically). Thus
  the relations matrix $T$ for $\cB_W$ is circulant, with all entries
  on its main diagonal being $-4$, bordered by two diagonals whose
  entries are all equal to $2$ (and all other entries $0$).  The
  matrix $T$ induces a directed graph $D$, as defined before the
  statement of 
  Theorem~\ref{thm:group_order}. (The graph $D$ is depicted in
  Figure~\ref{fig:example_D} in the case when $k=6$.) Now, $D$ clearly
  has tree number $k 2^{k-1}$. The first statement of the theorem now
  follows from Theorem~\ref{thm:group_order}, together with the fact
  that
\[
\cA_W\cong \bZ\oplus \cB_W\cong \bZ\oplus\bZ\oplus \cC_W.
\]

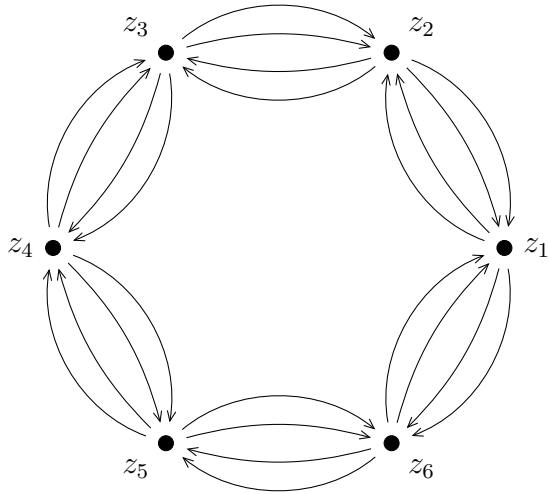
\begin{figure}
\begin{center}
\begin{tikzpicture}[scale=1.5,>=angle 45,shorten
  >=5,shorten <=5,
vertex/.style={circle,inner sep=2,fill=black,draw}]

\node at (0:2) [vertex,name=z0,label=east:$z_1$]{};
\node at (60:2) [vertex,name=z1,label=north east:$z_2$]{};
\node at (120:2) [vertex,name=z2,label=north west:$z_3$]{};
\node at (180:2) [vertex,name=z3,label=west:$z_4$]{};
\node at (240:2) [vertex,name=z4,label=south west:$z_5$]{};
\node at (300:2) [vertex,name=z5,label=south east:$z_6$]{};

\draw[->,bend left=15] (z0) edge (z1);
\draw[->,bend left=15] (z1) edge (z2);
\draw[->,bend left=15] (z2) edge (z3);
\draw[->,bend left=15] (z3) edge (z4);
\draw[->,bend left=15] (z4) edge (z5);
\draw[->,bend left=15] (z5) edge (z0);

\draw[->,bend left=40] (z0) edge (z1);
\draw[->,bend left=40] (z1) edge (z2);
\draw[->,bend left=40] (z2) edge (z3);
\draw[->,bend left=40] (z3) edge (z4);
\draw[->,bend left=40] (z4) edge (z5);
\draw[->,bend left=40] (z5) edge (z0);

\draw[->,bend left=15] (z0) edge (z5);
\draw[->,bend left=15] (z5) edge (z4);
\draw[->,bend left=15] (z4) edge (z3);
\draw[->,bend left=15] (z3) edge (z2);
\draw[->,bend left=15] (z2) edge (z1);
\draw[->,bend left=15] (z1) edge (z0);

\draw[->,bend left=40] (z0) edge (z5);
\draw[->,bend left=40] (z5) edge (z4);
\draw[->,bend left=40] (z4) edge (z3);
\draw[->,bend left=40] (z3) edge (z2);
\draw[->,bend left=40] (z2) edge (z1);
\draw[->,bend left=40] (z1) edge (z0);
\end{tikzpicture}
\end{center}
\caption{The directed graph $D$ when $k=6$.}
\label{fig:example_D}
\end{figure}
 
Since $\cB_W$ is generated by $x_1,x_2,\ldots,x_k$, by definition,
and since $\cA_W\cong \bZ\oplus \cB_W$, we see that the rank of
$\cA_W$ is at most $k+1$. So to prove the theorem, it suffices to show
that the rank of $\cA_W$ is at least $k+1$. In other words, we need to
show that the rank of $\cB_W$ is at least $k$.

Note that all the entries of the matrix $T$ are even, and so the
relations given by the rows of $T$ are all consequences of relations
$2x_i=0 \text{ for }i\in\{1,2,\ldots ,k\}$. So, if we write
$[2]\cB_W:=\{x+x:x\in \cB_W\}$, we see that $\cB_W/[2]\cB_W\cong
(\bZ_2)^k$. But this quotient group has rank $k$, and so the rank of
$\cB_W$ is at least $k$, and the theorem follows.
\end{proof}

We remark that there 
is no pair of disjoint edges $\{e_1,e_2\}$ with $\psi(e_1)=\psi(e_2)$
in the triangulation in
Theorem~\ref{thm:exponential}; nor are there distinct triangles
sharing the same three vertices. So the triangulation in this example
gives rise to a latin bitrade.

In the triangulation in Theorem~\ref{thm:exponential}, there are $t$
triangles of each colour, where $t=4k$. So the order of the torsion
subgroup $\cC_W$ of $\cA_W$ is exponential in $t$. Cavenagh and
Wanless~\cite[Corollary~5]{CavenaghWanless} show that for any
spherical latin bitrade with $t$ triangles of each colour,
$|\cC_W|<\frac{\sqrt{2}}{3}6^{(t-1)/3}$. The triangulation in
Theorem~\ref{thm:exponential} shows that an exponential upper bound on
the order of the torsion subgroup of $\cA_W$ is the best that could be
hoped for (though it is quite possible that the constant in the
exponential could be improved further).

We remarked in the proof of Theorem~\ref{thm:exponential} that the
rank of $\cA_W$ is always bounded above by $k+1$. So the rank of
$\cA_W$ in Theorem~\ref{thm:exponential} is maximal. Cavenagh and
Wanless~\cite[Corollary~7]{CavenaghWanless} construct spherical latin
bitrades where the rank of $\cA_W$ grows at least logarithmically in
$k$. Indeed, these bitrades have the extra property that whenever 
either partial latin square embeds in an abelian group $A$, the torsion rank of $A$ is
bounded below by a logarithmic function of $k$. It would be
interesting to know whether this logarithmic bound is close to being
the best possible.

Finally, we remark that the construction in
Theorem~\ref{thm:exponential} can be easily
generalised. Define a new triangulation, where each region $\cR_i$
consists of $2w$ black triangles, with $w$ of these triangles sharing
an edge with white triangles in $\cR_{i+1}$, and the remaining $w$
triangles sharing an edge with triangles in $\cR_{i-1}$. (We tile the
sphere with the regions $\cR_i$ in the same way as in
Theorem~\ref{thm:exponential}.) It is easy to show that for this
triangulation $|\cC_W|=k w^k$, and the group $\cB_W$ has a quotient
isomorphic to $(\bZ_w)^k$. In particular, this triangulation shows
that there is nothing special about the prime $2$ here: for any
prime $p$, there are examples of triangulations such that the
$p$-primary subgroups of $\cC_W$ have high rank.

\section{Comments}
\label{sec:comments}

This section is divided into two. The first subsection explores how
far our results extend to triangulations of surfaces of higher genus.
The second subsection states some questions and open problems
regarding $\cA_W$ for face $2$-colourable triangulations of the
sphere.

\subsection{Higher genus}

This section gives some examples that explore the extent to which
Theorem~\ref{thm:isomorphism} and Theorem~\ref{thm:group_order} can be
generalised.

The techniques in our paper depend crucially on the fact that the
vertices of our triangulation can be vertex $3$-coloured. For
triangulations of non-spherical surfaces, this is no longer true in
general, and so $2$-colourable triangulations no longer necessarilly
correspond to a biembedding of a pair of PTD$_\lambda(3,n)$.
Figure~\ref{fig:genus1_nasty_examples} gives three vertex
non-$3$-colourable triangulations of the torus; these examples show
that $\cA_W$ no longer always has free rank~$2$, and any
straightforward generalisation of Theorem~\ref{thm:group_order} is
unlikely. Moreover, consider the (vertex non-3-colourable)
nonorientable genus $3$
triangulation on vertices $\{0,1,2,3,4,5,6,7,8\}$ with white and black
triangles given by:
\begin{align*}
W&=\{012, 034, 057, 068, 135, 146, 178, 236, 247, 258\}\\
B&=\{013, 026, 047, 058, 124, 157, 168, 235, 278, 346\}.
\end{align*}
This example is due to Forbes, Grannell and Griggs~\cite{FGG}.  This
triangulation has $\cA_W\cong \bZ_3\oplus \bZ_6$ and $\cA_B\cong
\bZ_3\oplus\bZ_3$. Thus Theorem~\ref{thm:isomorphism} cannot extend to
this case. From Forbes et al, we can deduce that this
triangulation is a counterexample with the smallest number of
triangles if we restrict ourselves to having no 
pair of distinct edges $\{e_1,e_2\}$ with $\psi(e_1)=\psi(e_2)$ (in other words triangulations of simple graphs).

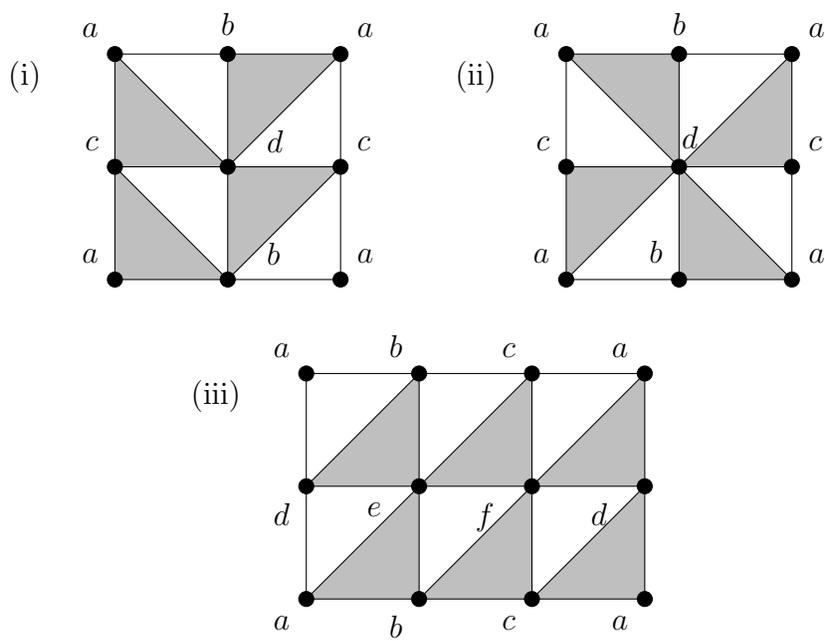
\begin{figure}
\begin{center}
\begin{tikzpicture}[fill=gray!50, scale=1.5,
vertex/.style={circle,inner sep=2,fill=black,draw}]

\draw (-0.8,1.8) node {(i)};

\filldraw (0,0) -- (1,0) -- (0,1) -- cycle;
\filldraw (1,0) -- (2,1) -- (1,1) -- cycle;
\filldraw (1,1) -- (2,2) -- (1,2) -- cycle;
\filldraw (0,1) -- (1,1) -- (0,2) -- cycle;
\draw (1,0) -- (2,0) -- (2,1) -- (2,2);
\draw (0,2) -- (1,2);

\node at (0,0) [vertex,label=north west:$a$] {};
\node at (1,0) [vertex,label=30:$\,\,\,\,b$] {};
\node at (2,0) [vertex,label=north east:$a$] {};
\node at (0,1) [vertex,label=north west:$c$] {};
\node at (1,1) [vertex,label=30:$\,\,\,\,d$] {};
\node at (2,1) [vertex,label=north east:$c$] {};
\node at (0,2) [vertex,label=north west:$a$] {};
\node at (1,2) [vertex,label=north:$b$] {};
\node at (2,2) [vertex,label=north east:$a$] {};

\draw (3.2,1.8) node {(ii)};

\filldraw (4,0) -- (5,1) -- (4,1)  -- cycle;
\filldraw (5,0) -- (6,0) -- (5,1) -- cycle;
\filldraw (5,1) -- (5,2) -- (4,2) -- cycle;
\filldraw (5,1) -- (6,1) -- (6,2) -- cycle;
\draw (4,0) -- (5,0);
\draw (6,0) -- (6,1);
\draw (4,1) -- (4,2);
\draw (5,2) -- (6,2);

\node at (4,0) [vertex,label=north west:$a$] {};
\node at (5,0) [vertex,label=north west:$b$] {};
\node at (6,0) [vertex,label=north east:$a$] {};
\node at (4,1) [vertex,label=north west:$c$] {};
\node at (5,1) [vertex,label=north:$\,\,\,\,d$] {};
\node at (6,1) [vertex,label=north east:$c$] {};
\node at (4,2) [vertex,label=north west:$a$] {};
\node at (5,2) [vertex,label=north:$b$] {};
\node at (6,2) [vertex,label=north east:$a$] {};

\end{tikzpicture}
\end{center}

\begin{center}
\begin{tikzpicture}[fill=gray!50, scale=1.5,
vertex/.style={circle,inner sep=2,fill=black,draw}]

\draw (-0.8,1.8) node {(iii)};

\filldraw (0,0) -- (1,0) -- (1,1) -- cycle;
\filldraw (1,0) -- (2,0) -- (2,1) -- cycle;
\filldraw (2,0) -- (3,0) -- (3,1) -- cycle;
\filldraw (0,1) -- (1,1) -- (1,2) -- cycle;
\filldraw (1,1) -- (2,1) -- (2,2) -- cycle;
\filldraw (2,1) -- (3,1) -- (3,2) -- cycle;

\draw (0,0) -- (0,1) -- (0,2) -- (1,2) -- (2,2) -- (3,2);

\draw (0,0) node [vertex,label=south west:$a$]{};
\draw (1,0) node [vertex,label=south west:$b$]{};
\draw (2,0) node [vertex,label=south west:$c$]{};
\draw (3,0) node [vertex,label=south west:$a$]{};
\draw (0,1) node [vertex,label=south west:$d$]{};
\draw (1,1) node [vertex,label=south west:$e\,\,\,\,$]{};
\draw (2,1) node [vertex,label=south west:$f\,\,\,\,$]{};
\draw (3,1) node [vertex,label=south west:$d\,\,\,\,$]{};
\draw (0,2) node [vertex,label=north west:$a$]{};
\draw (1,2) node [vertex,label=north west:$b$]{};
\draw (2,2) node [vertex,label=north west:$c$]{};
\draw (3,2) node [vertex,label=north west:$a$]{};

\end{tikzpicture}
\end{center}
\caption{Three triangulations of the torus, with $\cA_W\cong \bZ_3$,
  $\cA_W\cong \bZ\oplus\bZ$ and $\cA_W\cong\bZ_9$ respectively.}
\label{fig:genus1_nasty_examples}
\end{figure}

We now restrict ourselves to face $2$-colourable triangulations that
are also vertex $3$-colourable. (Such triangulations give rise to biembeddings of 
PTD$_\lambda(3,n)$.) We note that the vertex colouring induces
a consistent orientation on the faces of the triangulation, so
the surfaces we consider are all orientable.

Figure~\ref{fig:non_isomorphism_example} gives an example of a vertex
$3$-colourable triangulation of the torus with $\cA_W\cong
\bZ\oplus\bZ\oplus\bZ_6$ and $\cA_B\cong
\bZ\oplus\bZ\oplus\bZ_3$. This example is due to Cavenagh and
Wanless~\cite[Example~5.1]{CavenaghWanless}. Thus
$\cA_W\not\cong\cA_B$ in this example, and so
Theorem~\ref{thm:isomorphism} does not hold in general for vertex
$3$-colourable surfaces of higher genus.

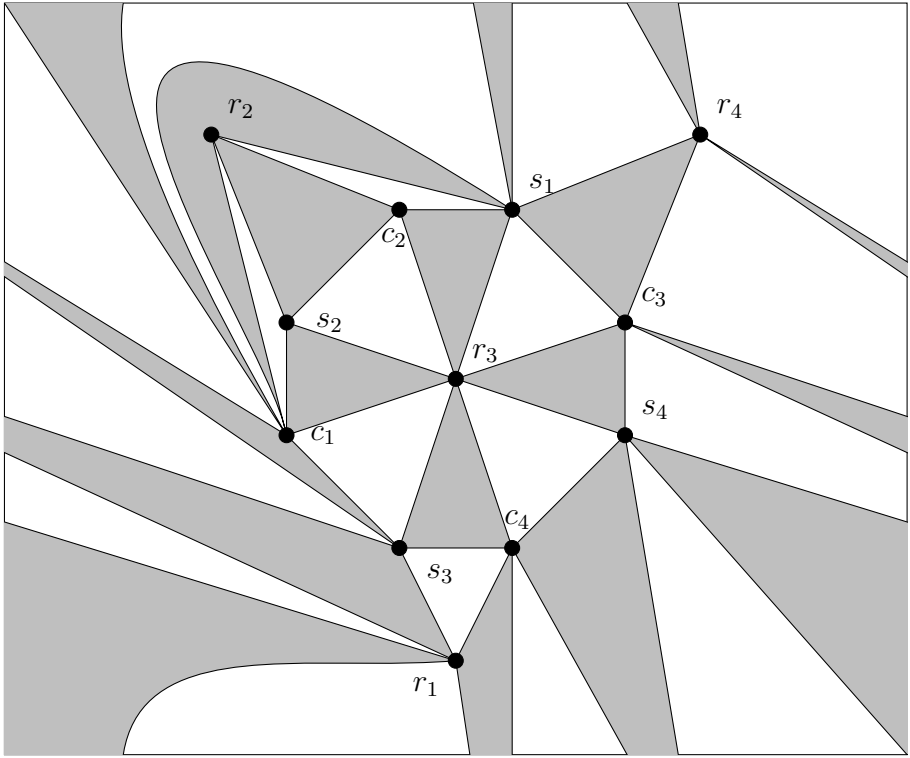
\begin{figure}
\begin{center}
\begin{tikzpicture}[fill=gray!50, scale=0.5,
vertex/.style={circle,inner sep=2,fill=black,draw}]

\coordinate (v1) at (0,-7.5);
\coordinate (v2) at (-6.5,6.5);
\coordinate (v3) at (0,0);
\coordinate (v4) at (6.5,6.5);

\coordinate (va) at (-4.5,-1.5);
\coordinate (vb) at (-1.5,4.5);
\coordinate (vc) at (4.5,1.5);
\coordinate (vd) at (1.5,-4.5);

\coordinate (ve) at (1.5,4.5);
\coordinate (vf) at (-4.5,1.5);
\coordinate (vg) at (-1.5,-4.5);
\coordinate (vh) at (4.5,-1.5);

\coordinate (v1_up) at (0,12.5);
\coordinate (v1_right) at (24,-7.5);
\coordinate (v4_down) at (6.5,-13.5);
\coordinate (v4_left) at (-17.5,6.5);
\coordinate (va_right) at (19.5,-1.5);
\coordinate (va_down) at (-4.5,-21.5);
\coordinate (vc_left) at (-19.5,1.5);
\coordinate (vd_up) at (1.5,15.5);
\coordinate (vg_right) at (22.5,-4.5);
\coordinate (vh_left) at (-19.5,-1.5);
\coordinate (vh_up) at (4.5,18.5);
\coordinate (ve_down) at (1.5,-17.5);

\clip [draw] (-12,-10) rectangle (12,10);

\filldraw (v3) -- (vh) -- (vc) -- cycle;
\filldraw (v3) -- (ve) -- (vb) -- cycle;
\filldraw (v3) -- (vf) -- (va) -- cycle;
\filldraw (v3) -- (vd) -- (vg) -- cycle;
\filldraw (vf) -- (vb) -- (v2) -- cycle;
\filldraw (vc) -- (ve) -- (v4) -- cycle;

\filldraw (v2) -- (va) .. controls +(-0.5,3) and (-15,15) .. (ve) -- cycle;

\filldraw (v1) -- (vd) -- (ve_down) -- cycle;
\filldraw (v1_up) -- (vd_up) -- (ve) -- cycle;

\filldraw (vd) -- (vh) -- (v4_down) -- cycle;
\filldraw (vd_up) -- (vh_up) -- (v4) -- cycle;

\filldraw (v4_left) -- (va) -- (vg) -- cycle;
\filldraw (v4) -- (va_right) -- (vg_right) -- cycle;

\filldraw (vg) -- (v1) -- (vc_left) -- cycle;
\filldraw (vg_right) -- (v1_right) -- (vc) -- cycle;

\filldraw (vh) -- (v1_right) -- (12,-10) -- cycle;

\filldraw (va) .. controls (-14,15) and (-6,12) .. (v1_up) -- ++(0,10)
-- (-12,10) -- cycle;
\filldraw (va_down) .. controls (-14,-5) and (-6,-8) .. (v1) --
(vh_left) -- (va_down) -- cycle;

\node at (v3) [vertex,label=north east:$r_3$]{};
\node at (ve) [vertex,label=north east:$s_1\,$]{};
\node at (vb) [vertex,label=south:$c_2\,\,$]{};
\node at (vc) [vertex,label=north east:$c_3$]{};
\node at (vh) [vertex,label=north east:$s_4$]{};
\node at (vd) [vertex,label=north:$\,\,c_4$]{};
\node at (vg) [vertex,label=south east:$\,\,s_3$]{};
\node at (va) [vertex,label=east:$\,c_1$]{};
\node at (vf) [vertex,label=east:$\,\,s_2$]{};

\node at (v4)  [vertex,label=north east:$r_4$]{};
\node at (v2) [vertex,label=north east:$r_2$]{};
\node at (v1)  [vertex,label=south west:$r_1$]{};

\end{tikzpicture}
\end{center}
\caption{A face 2-coloured triangulation of the torus with $\cA_W\cong
  \bZ\oplus\bZ\oplus\bZ_6$ and $\cA_B\cong \bZ\oplus \bZ\oplus \bZ_3$.}
\label{fig:non_isomorphism_example}
\end{figure}

Nevertheless, the techniques of our paper do provide a weaker result
in this more general setting. Let $\cG$ be a vertex $3$-coloured
triangulation of some (orientable) surface, with faces $2$-coloured
black and white. Let $R$, $C$ and $S$ be the colour
classes of the vertex $3$-colouring. Define the groups $\cA_W$,
$\cA_B$, $\cB_W$ and $\cB_B$ as before. Our analysis of the structure
of $\cB_W$ and $\cB_B$ does not depend on the genus of the underlying
surface (and so in particular, $\cB_W\cong \cB_B\cong \bZ\oplus \cC$,
where $\cC$ is finite). Indeed, the only time we used the fact that
the underlying surface is a sphere was in
Lemma~\ref{lem:cycle_lemma}. For a
graph $\Gamma$ embedded in an orientable surface of arbitrary genus,
the statement of this lemma is still true if we replace the final isomorphism
by the statement that $\langle \tilde{E}\rangle$ is a quotient of the group
\[
F(E)/
  \left\langle\sum_{a=1}^{\ell_i}(-1)^ae_{ia}:i\in\{1,2,\ldots,k\}\right\rangle.
\]
Thus we may argue, following the proofs of
Theorems~\ref{thm:isomorphism} and \ref{thm:group_order}, that the
following theorem is true.
\begin{theorem}
\label{thm:highergenus}
Let $\cG$ be a face $2$-coloured vertex $3$-coloured triangulation of
an orientable surface. Then $\cB_W\cong \cB_B$. Moreover,
$\cA_W\cong\bZ\oplus \bZ\oplus M$ and $\cA_B\cong\bZ\oplus \bZ \oplus N$,
where $M$ and $N$ are quotients of the (finite) torsion subgroup
$\cC_W$ of $\cB_W$.
\end{theorem}

\subsection{Some open problems}

How large can the torsion subgroup $\cC_W$  of $\cA_W$ be? More precisely:

\begin{question}
\label{qn:order}
Let $m_t$ be the maximal order of $\cC_W$, over all $2$-face colourable
triangulations of the sphere with $t$ triangles of each colour. What
is the value of $\limsup_{t\rightarrow\infty} (m_t)^{1/t}$?
\end{question}

Theorem~\ref{thm:exponential} and~\cite[Corollary~5]{CavenaghWanless}
combine to show that this limit lies strictly between $1.189$ and
$1.818$. In fact, the lower bound can be improved to $1.201$ by using
twelve rather than eight triangles in each region $\cR_i$ (see the final
paragraph of the previous section, with $w=3$). Question~\ref{qn:order}
asks whether this value can be determined precisely.

Both partial latin squares in a spherical latin bitrade $(W,B)$ embed
in the finite group $\cC_W$, but there is no reason to suppose that
$\cC_W$ is the group of minimal order with this property (though a
minimal group must be isomorphic to a quotient of $\cC_W$, by
Theorem~\ref{thm:small_universe}). Moreover, we believe that there
should be examples of spherical latin bitrades $(W,B)$ such that
$W$ and $B$ both embed in some abelian group of rank smaller than
$\cC_W$. Example~5.2 in Cavenagh and Wanless~\cite{CavenaghWanless} is
a spherical latin bitrade $(W,B)$ with $\cC_W\cong \bZ_2\oplus
\bZ_6$ (so has rank $2$), where $W$ embeds in $\bZ_6$ (of rank
$1$). But the smallest abelian group in which $B$ embeds in this
example is $\cC_W$ itself.

\begin{question}
\label{qn:bitrade_order}
Is there a family of bitrades with the property that the minimum order
of an abelian group in which both mates embed is exponential in the
size of the bitrade?
\end{question}
\begin{question}
\label{qn:bitrade_rank}
Is there a family of bitrades where the
minimal rank of an abelian group in which both mates embed is linear
in the size of the bitrade? 
\end{question}

It is clear that the order of an abelian group in which both mates
embed is at least linear in the size of the
bitrade;~\cite[Corollary~7]{CavenaghWanless} shows that the minimal
rank of an abelian group in which both mates embed can be at least
logarithmic in the size of the bitrade.


\begin{thebibliography}{99}
\bibitem{Biggs} Norman Biggs, `Algebraic potential theory on graphs',
  \emph{Bull. London Math. Soc.} \textbf{29} (2007), 641--682.
\bibitem{BondyMurty} John A. Bondy and U. S. R. Murty, \emph{Graph Theory}, Graduate Texts in Mathematics, 244. Springer, New York, 2008.
\bibitem{Cameron} Peter J. Cameron, `Research problems from the
  BCC22', \emph{Discrete Math.} \textbf{311} (2011), 1074--1083.
\bibitem{Cavenagh} Nicholas J. Cavenagh, `Embedding 3-homogeneous latin trades into abelian 2-groups', \emph{Comentat. Math. Univ. Carolin.} \textbf{45} (2004), 194--212.
\bibitem{CavenaghSurvey} Nicholas J. Cavenagh, `The theory and application of latin bitrades: a survey', \emph{Math. Slovaca} \textbf{58} (2008), 691--718.  
\bibitem{CavenaghWanless} Nicholas J. Cavenagh and Ian M. Wanless,
`Latin trades in groups defined on planar triangulations',
\emph{J. Algebr. Comb.} \textbf{30} (2009), 323--347.
\bibitem{CavenaghDrapal} Ale\v{s} Dr\'apal and Nick Cavenagh, Open Problem~8, `Open problems from Workshop on latin trades',
  Prague, 6-10 February
  2006. \url{http://www.karlin.mff.cuni.cz/~rozendo/op.html}
\bibitem{CRC} Charles J. Colbourn and Jeffery H. Dinitz, (eds.), Second Edition, \emph{The CRC Handbook of Combinatorial Designs}, Chapman \& Hall/CRC Press, Boca Raton (2006).
\bibitem{Drapal} Al\v{e}s Dr\'apal, `Hamming distances of groups and quasigroups', \emph{Discrete Math.} \textbf{235} (2001) 189--197.
\bibitem{Drapal2} Al\v{e}s Dr\'apal, `On elementary moves that generate all spherical Latin trades', \emph{Comment. Math. Univ. Carolin.} \textbf{50} (2009), no. 4, 477--511.
\bibitem{Drapal3} Al\v{e}s Dr\'apal, `Geometrical structure and construction of Latin trades', \emph{Adv. Geom.} \textbf{9} (2009), no. 3, 311--348.
\bibitem{DHK} Al\v{e}s Dr\'apal, Carlo H\"am\"al\"ainen and
  V\'it\v{e}zslav Kala, `Latin bitrades, dissections of equilateral
  triangles, and abelian groups', \emph{J. Comb. Des.}, \textbf{18}
  (2010), 1--24.
\bibitem{newDK} Ale\v{s} Dr\'apal and Tom\'a\v{s} Kepka, `Group modifications of some partial groupoids', \emph{Annals of Discr. Math.} \textbf{18} (1983), 319--332.
\bibitem{DrapalKepka} Ale\v{s} Dr\'apal and Tom\'a\v{s} Kepka,
  `Exchangeable partial groupoids I', \emph{Acta Universitatis
    Carolinae - Mathematica et Physica} \textbf{24} (1983), 57--72.
\bibitem{FGG} Anthony D. Forbes, Mike J. Grannell and Terry S. Griggs,
`Configurations and trades in Steiner triple systems', \emph{Australasian
J. Comb.} \textbf{29} (2004), 75--84.
\bibitem{Heawood} Percy J. Heawood, `On the four-colour map theorem',
  \emph{Quart J. Pure Math.} \textbf{29} (1898), 270--285.
\bibitem{Kourovka17} V.D. Mazurov and E.I. Khukhro (eds), \emph{The
  Kourovka notebook (Unsolved Problems in group theory)}, 17th
edition, Russian Academy of Sciences, Siberian Division, Institute of
Mathematics, Novosibirsk, 2010.
\bibitem{Rotman} Joseph J. Rotman, \emph{An Introduction to the Theory
    of Groups} (3rd edition), Wm. C. Brown Publishers, Dubuque, Iowa, 1988.
\bibitem{Tutte} William T. Tutte, \emph{Graph theory},  Addison--Wesley,
  Reading, MA, 1984.
\bibitem{Wilson} Robert A. Wilson, \emph{Graphs, colourings and the
    four-colour theorem}, Oxford University Press, Oxford, 2002.
    
\end{thebibliography}
\end{document}